\documentclass{amsart}
\usepackage{amsfonts,amssymb,amsmath,amsthm}
\usepackage{url}
\usepackage{enumerate}
\newtheorem{thm}{Theorem}[section]
\newtheorem{cor}[thm]{Corollary}
\newtheorem{lem}[thm]{Lemma}
\newtheorem{prop}[thm]{Proposition}
\theoremstyle{definition}

\theoremstyle{remark}
\newtheorem{rem}[thm]{\bf Remark}

\numberwithin{equation}{section}

\begin{document}
\title{Complex dimension of additive subgroups of $\mathbb{R}^{n}$}

\author{Yahya N'dao and Adlene Ayadi}

 \address{Yahya N'dao, University of Moncton, Department of mathematics and statistics, Canada}
 \email{yahiandao@yahoo.fr }
\address{Adlene Ayadi, University of Gafsa, Faculty of sciences, Department of Mathematics,Gafsa, Tunisia.}

\email{adlenesoo@yahoo.com}

\thanks{This work is supported by the research unit: syst\`emes dynamiques et combinatoire:
99UR15-15} 

\begin{abstract}
In this paper, we define the complex dimension of any additive
subgroup of $\mathbb{R}^{n}$ which generalize the euclidien
dimension given for the vector space. We give an explicit method
to calculate this dimension.
\end{abstract}
\maketitle

\section{\bf Introduction }
The additive groups are seen as weaker than vector space
structures. The stability by scalar multiplication in an additive
group is not totally or non-existent in general. The linear
independence defined the dimension of vector spaces. By analogy,
we define the dimension of a discrete additive group as a
dimension of the vector space that can generates.  As if we assume
that it is stable by virtual scalar multiplication. For this
reason, this dimension is given purely complex. The difference
between these two structures is given by the complex dimension
which will be defined in the following.

In \cite{mW} M.Waldschmidt gave the form of any closed subgroup
$F$ of $\mathbb{R}^{n}$ as $F=E+D$ with $E$ is a vector space and
$D$ is a discrete additive group. By Theorem 2.1, in \cite{mW},
this means that there is a basis $(u_{1},\dots, u_{n})$ of
$\mathbb{R}^{n}$ and $0\leq r, p\leq n$ with $p+r\leq n$ such that
$$F:=\left\{\begin{array}{c}
  \underset{k=1}{\overset{p}{\sum}}\mathbb{R}u_{k}+
\underset{k=p+1}{\overset{p+r}{\sum}}\mathbb{Z}u_{k},\ \
\mathrm{if}\ p>0\ \mathrm{and}\ r>0\\
  \underset{k=1}{\overset{p}{\sum}}\mathbb{R}u_{k},\ \
\mathrm{if}\ r=0 \ \ \ \ \ \ \ \  \ \ \ \ \ \ \ \ \ \ \ \ \ \ \ \
\ \ \ \
\\ \underset{k=1}{\overset{r}{\sum}}\mathbb{Z}u_{k},\ \
\mathrm{if}\ p=0\ \ \ \ \ \ \ \ \ \ \ \ \ \ \ \ \ \ \ \ \ \ \ \ \
\ \ \
\end{array}\right.$$
 Let $W$ be
the vector space generated by $D$. We define the complex dimension
of $F$ as $\widetilde{\mathrm{dim}}(F)=\mathrm{dim}(E)+i
\mathrm{dim}(W)$ (i.e. $\widetilde{\mathrm{dim}}(F)=p+ir$). We
want to define a new dimension of the discrete groups of
$\mathbb{R}^n$ which generalizes that given as a manifold. This
dimension can not be rational for the discrete group, for example
if $u,v$ are free in $\mathbb{R}^{n}$ then
$dim(\mathbb{Z}u)=dim(\mathbb{Z}u+\mathbb{Z}v)=0$  as manifolds,
We can not distinguish between them, but with the complex
dimension we have $\widetilde{\mathrm{dim}}(\mathbb{Z}u)=i$ and
$\widetilde{\mathrm{dim}}(\mathbb{Z}u + \mathbb{Z}v) = 2i$. For
any subset $A$ of $\mathbb{R}^{n}$, denote by $vect(A)$ the vector
subspace of $\mathbb{R}^{n}$ generated by $A$. Therefore the
dimension of any closed subgroup of $\mathbb{R}^{n}$ (as a
manifold) is equal the real part of the complex dimension. As a
manifold we can verify that
 $$dim(F)=\mathrm{Max}\{dim(V);\ V\
\mathrm{vector}\ \mathrm{ space}\ V\subset F\}.$$  For any
additive subgroup $H$ of $\mathbb{R}^{n}$, we call the
\emph{complex dimension} of $H$ the number
$\widetilde{\mathrm{dim}}(H):=p+i(q-p)$ where
$$p=\mathrm{Max}\{dim(V);\ V\ \mathrm{vector}\ \mathrm{ space}\
V\subset H\}\ \ \ \ \mathrm{and}$$ $$q=Min\{dim(V);\ V\
\mathrm{vector}\ \mathrm{ space}\ H\subset V\}.\ \ \ \ \ \ \ \ $$
This means that $q=dim(vect(H))$.

 For a number $z\in\mathbb{C}$, we write $z = \Re(z)+ i
\Im(z)$, where $\Re(z)$ and $\Im(z)\in \mathbb{R}$. \ \\ The
$\Re(\widetilde{\mathrm{dim}}(H))$ represents the dimension of the
vector space generated by all points of $H$ which are stable by
scalar multiplication and $\Im(\widetilde{\mathrm{dim}}(H))$
represents the dimension of the vector space generated by all
points of $H$ which are stable only by addition.

 Our principal results are
the following:
\medskip

\begin{thm}\label{t:1} Let $H$ be an additive subgroup of
$\mathbb{R}^{n}$. If $\overline{H}=E+D$ with $E$ is a vector space
and $D$ is a discrete additive group such that $vect(D)\oplus
E=vect(H)$, then $H=H_{1}+D$ with $H_{1}$ is an additive subgroup
dense in $E$.
\end{thm}
\medskip

\begin{cor}\label{C:CC01} If $\Im(\widetilde{\mathrm{dim}}(H))\neq
0$ then $H$ can not be dense in $\mathbb{R}^{n}$.
\end{cor}
\medskip

\begin{cor}\label{Ct:1} Let $H$ and $K$ be two additive subgroups of
$\mathbb{R}^{n}$. One has:\ \\ (i) If
$\widetilde{\mathrm{dim}}(H)=p+ir$ then
$\widetilde{\mathrm{dim}}(\overline{H})=p'+ir$ with $p'\geq p$. \
\\ (ii) If $K\subset H$ then
$|\widetilde{\mathrm{dim}}(K)|\leq
|\widetilde{\mathrm{dim}}(H)|$.\
\end{cor}
\medskip

\begin{rem}\label{c:1} (i) $H$ is closed and discrete if and only if
$\Re(\widetilde{\mathrm{dim}}(H))=0$.\ \\ (ii) If $E$ is a vector
space then $\widetilde{\mathrm{dim}}(E)=dim(E)$.\ \\ (iii) If $E$
is a connected component of $\overline{H}$ containing $0$ then
$\widetilde{\mathrm{dim}}(E)\in \mathbb{N}$.
\end{rem}
\bigskip

We will use the following notations and definitions to
characterize the density of any subgroup $H$ of $\mathbb{R}^{n}$.\
\\
Let $u_{1},\dots, u_{p}\in \mathbb{R}^{n}$, $p\geq n+1$. Suppose
that $(u_{1},\dots, u_{n})$ be a basis of $\mathbb{R}^{n}$ and
$u_{k}=\underset{i=1}{\overset{n}{\sum}}\alpha_{k,i}u_{i}$ such
that $1,\alpha_{k,k_{1}}\dots,\alpha_{k,k_{r_{k}}}$ is the longest
list of $\{1,\alpha_{k,1},\dots, \alpha_{k,n}\}$ that are
rationally independent. Denote by $I_{k}=\{k_{i}:\ \  1\leq i \leq
r_{k}\}$. Then for every  $j\notin I_{k}$, there exists
$\gamma^{(k)}_{j,1},\dots,
\gamma^{(k)}_{j,r_{k}},t_{k}\in\mathbb{Q}$ such that
$\alpha_{k,j}=t_{k,j}+\underset{i\in I_{k}
}{\sum}\gamma^{(k)}_{j,i}\alpha_{k,i}$.

 We obtain:
\begin{align*} u_{k}& =
\underset{j=1}{\overset{n}{\sum}}\alpha_{k,j}u_{j}\\ \ & =
\underset{j\in I_{k}}{\sum}\alpha_{k,j}u_{j}+ \underset{j\notin
I_{k}}{\sum}\left(t_{k,j}+\underset{i\in
I_{k}}{\sum}\gamma^{(k)}_{j,i}\alpha_{k,i}\right)u_{j}
\\ \ & =\underset{j\in
I_{k}}{\sum}\alpha_{k,j}u_{j}+ \underset{i\in
I_{k}}{\sum}\left(\underset{j\notin
I_{k}}{\sum}\gamma^{(k)}_{j,i}\alpha_{k,i}u_{j}\right)+\underset{j\notin
I_{k}}{\sum}t_{k,j}u_{j}\\ \ & =\underset{j\in
I_{k}}{\sum}\alpha_{k,j}\left(u_{j}+ \underset{i\notin
I_{k}}{\sum}\gamma^{(k)}_{i,j}u_{i}\right)+\underset{j\notin
I_{k}}{\sum}t_{k,j}u_{j}\\
\end{align*}

Let $q\in \mathbb{N}^{*}$ and $m^{(k)}_{j,i},p_{k,j}\in
\mathbb{Z}$ such that $t_{k,j}=\frac{p_{k,j}}{q}$ and
$\gamma^{(k)}_{i,j}=\frac{m^{(k)}_{i,j}}{q}$ for every $k\notin
I_{k}$. Therefore,
\begin{equation}\label{eq02}
qu_{k}=\underset{j\in
I_{k}}{\sum}\alpha_{k,j}\left(qu_{j}+\underset{i\notin I_{k}
}{\sum}m^{(k)}_{i,j}u_{i}\right)+\underset{j\notin I_{k}
}{\sum}p_{k,j}u_{j}.
\end{equation}

Denote by $u'_{k,j}= qu_{j}+\underset{i\notin I_{k}
}{\sum}m^{(k)}_{i,j}u_{i}$ for every $n+1\leq k\leq p$ and $j\in
I_{k}$. See that
\begin{equation}\label{eq04}
u'_{k,j}\in \underset{i=1}{\overset{n}{\sum}}\mathbb{Z}u_{i}\ \ \
\end{equation}

Denote by $F$ the vector space generated by all $u'_{k,j}$, $j\in
I_{k}$, $n+1\leq k \leq p$.  Define the matrix $M_{H}$ formed by
all $u'_{k,j}$, $j\in I_{k}$, $n+1\leq k \leq p$ as colon vectors.

\bigskip

\
\\
{\bf Example:1.} Let $H=\mathbb{Z}u_{1}+\dots +\mathbb{Z}u_{7}$
with $u_{1}=[1, 0, 0]^{T}$, $u_{2}=[0,1, 0]^{T}$,
$u_{3}=[0,0,1]^{T}$, $u_{4}=[1,\sqrt{2}, 1]^{T}$, $u_{5}=[0, 1,
\sqrt{3}]^{T}$, $u_{6}=[\sqrt{2}, \sqrt{3},1]^{T}$, $u_{7}=[1,
\sqrt{2}, \sqrt{2}]^{T}$.  For : \ \\ - $k=4$, we have
$\{\sqrt{2}\}$ with $1$ forms the  longest list of
$\{1,1,\sqrt{2}, 1\}$ that are rationally independent. \ \\ -
$k=5$, we have $\{\sqrt{3}\}$ with $1$ forms the  longest list of
$\{1,0,1, \sqrt{3}\}$ that are rationally independent. \ \\ -
$k=6$, we have $\{\sqrt{2}, \sqrt{3}\}$ with $1$ forms the longest
list of $\{1,\sqrt{2},\sqrt{3}, 1\}$ that are rationally
independent.\ \\ - $k=7$, we have $\{\sqrt{2}\}$ with $1$ forms
the longest list of $\{1,1,\sqrt{2}, \sqrt{2}\}$ that are
rationally independent.\ \\
\\ We can choose $I_{4}=\{2\}$, $I_{5}=\{3\}$, $I_{6}=\{1,2\}$ and
$I_{7}=\{2\}$. We obtain:

\begin{align*}
  u_{4} & = \sqrt{2}u_{2}+ (u_{1}+u_{3}) &   \ \ \  I_{4}=\{2\}\\
   u_{5} &=\sqrt{3}u_{3}+ (u_{2}) & \ \ \ \ I_{5}=\{3\} \\
   u_{6}& =\sqrt{2}u_{1}+\sqrt{3}u_{2}+ (u_{3}) & \ \ \ \ I_{6}=\{1,2\} \\
   u_{7}& =\sqrt{2}(u_{2}+u_{3})+ (u_{1})&\ \ \ \ I_{7}=\{2\}\
\end{align*}

Therefore
\begin{align*}
  u'_{4,2} & =u_{2}=[0,1,0]^{T}  \\
   u'_{5,3} &=u_{3}=[0,0,1]^{T}  \\
   u'_{6,1}& =u_{1}=[1,0,0]^{T}\\
   u'_{6,2}& =u_{2}=[0,1,0]^{T}\\
   u'_{7,2}& =u_{2}+u_{3}=[0,1,1]^{T} \
\end{align*}
Then $$M_{H}=\left[\begin{array}{ccccc}
  0 & 0 & 1 & 0 &0\\
  1 & 0 & 0 & 1 &1\\
  0 & 1 & 0 & 0& 1
\end{array}\right]$$ Since $rank(M_{H})=3$ then $dim(F)=3$ and so
$F=\mathbb{R}^{3}$.

\bigskip

\
\\
{\bf Example:2} Let $H=\mathbb{Z}u_{1}+\dots +\mathbb{Z}u_{7}$
with  $u_{1}=[1, 0, 0]^{T}$, $u_{2}=[0,1, 0]^{T}$,
$u_{3}=[0,0,1]^{T}$, $u_{4}=[1,\sqrt{2}, 1]^{T}$,
$u_{5}=[\sqrt{2}, 1, \sqrt{2}]^{T}$, $u_{6}=[2\sqrt{2},
2,3\sqrt{2}]^{T}$, $u_{7}=[1, 3\sqrt{2}, \sqrt{2}]^{T}$. For every
$4\leq k \leq 7$ we have $\{\sqrt{2}\}$ with $1$ forms the longest
list of $\{1,\alpha_{k,1},\alpha_{k,2}, \alpha_{k,3}\}$ that are
rationally independent. We can choose $I_{4}=\{2\}$,
$I_{5}=\{1\}$, $I_{6}=\{1\}$ and $I_{7}=\{2\}$. We obtain:

\begin{align*}
  u_{4} & = \sqrt{2}u_{2}+ (u_{1}+u_{3}) &   \ \ \  I_{4}=\{2\}\\
   u_{5} &=\sqrt{2}(u_{1}+ u_{3})+ (u_{2}) & \ \ \ \ I_{5}=\{1\} \\
   u_{6}& =\sqrt{2}(2u_{1}+3u_{3})+ (2u_{2}) & \ \ \ \ I_{6}=\{1\} \\
   u_{7}& =\sqrt{2}(3u_{2}+u_{3})+ (u_{1})&\ \ \ \ I_{7}=\{2\}\
\end{align*}

Therefore

\begin{align*}
  u'_{4,2} & =u_{2}=[0,1,0]^{T} \\
   u'_{5,1} &=u_{1}+u_{3}=[1,0,1]^{T}\\
   u'_{6,1}& =2u_{1}+3u_{3}=[2,0,3]^{T}\\
   u'_{7,2}& =3u_{2}+u_{3}=[0,3,1]^{T}
\end{align*}
Then $$M_{H}=\left[\begin{array}{cccc}
  0 & 1 & 2 & 0 \\
  1 & 0 & 0 & 3 \\
  0 & 1 & 3 & 1
\end{array}\right]$$ Since $rank(M_{H})=3$ then $dim(F)=3$ and so
$F=\mathbb{R}^{3}$.
\
\bigskip

Denote by $L(M_{H})=rank(M_{H})$.
\medskip

\begin{thm}\label{t:2} Let $H$ be an additive subgroup of
$\mathbb{R}^{n}$ generated by $u_{1},\dots, u_{m}$. Then
$\widetilde{\mathrm{dim}}(\overline{H})= L(M_{H})+i(q-L(M_{H}))$,
where $q=dim(vect(H))$.
\end{thm}
\medskip

\begin{thm}\label{t:3} Let $H$ be an additive subgroup of
$\mathbb{R}^{n}$ with complex dimension $p+ir$. Then there exists
$u\in \mathbb{R}^{n}$ such that $H+\mathbb{Z}u$ is dense in
$vect(H)$. (i.e.
$\widetilde{\mathrm{dim}}(\overline{H+\mathbb{Z}u})=p+r$).
\end{thm}
\medskip

\begin{cor}\label{c:3}  Let $H$ be an additive subgroup of
$\mathbb{R}^{n}$ with complex dimension $p+ir$. If $p+r<n$ then
for every $u\in \mathbb{R}^{n}$, $H+\mathbb{Z}u$ can not be dense
in $\mathbb{R}^{n}$.
\end{cor}
\medskip

 Let $H=E+D$ and $K=E'+D'$ be two additive subgroup of
$\mathbb{R}^{n}$, where $E$ and $E'$  are two vector spaces, $D$
and $D'$ are two discrete groups. A map $f: H\longrightarrow K$
called \emph{homomorphism of closed additive group} if
$f=f_{1}\oplus f_{2}$ with $f_{1}: E\longrightarrow E'$ is a
linear map and $f_{2}: D\longrightarrow D'$ is a homomorphism of
group,  (i.e. $f(\lambda x+ p y)= \lambda f_{1}(x)+ p f_{2}(y)$
for every $\lambda\in \mathbb{R}$, $p\in \mathbb{Z}$, $x\in E$ and
$y\in D$). An homomorphism of closed additive group is called
\emph{isomorphism} of closed additive group if it is invertible.

\begin{thm}\label{t:5}  Let $H$ and $K$ be two closed additive subgroups of
$\mathbb{R}^{n}$ and $f:\  H\longrightarrow K$ be an homomorphism
of closed additive group. One has:\ \\ (i) If $f$ is injective
then $|\widetilde{\mathrm{dim}}(H)|\leq
|\widetilde{\mathrm{dim}}(K)|$.\ \\ (ii) If $f$ is surjective then
$|\widetilde{\mathrm{dim}}(H)|\geq
|\widetilde{\mathrm{dim}}(K)|$.\ \\ (iii) If $f$ is invertible
then $\widetilde{\mathrm{dim}}(H)= \widetilde{\mathrm{dim}}(K)$.\
\\ (iv)  $f(H)$ is a closed additive subgroup of $K$.\ \\
 (v)  $f^{-1}(L)$ is  a closed additive subgroup of $H$, for
every closed subgroup of $K$.
\end{thm}

\bigskip

\section{{\bf Proof of Theorem ~\ref{t:1} and Corollaries  ~\ref{C:CC01}, ~\ref{Ct:1}}}

\begin{lem}\label{L:1}$($\cite{mW}, Theorem 2.1$)$  Let $H$ be an additive subgroup of $\mathbb{R}^{n}$. Then
 there exist a vector space $E$ and a discrete additive group $D$ such that $\overline{H}=E+D$ with
 $E\oplus vect(D)=vect(H)$.
\end{lem}
\medskip

\begin{proof}[Proof of Theorem~\ref{t:1}]  By  Lemma ~\ref{L:1}, we can write
$\overline{H}=E+D$ with $E$ is a vector space and $D$ is a
discrete additive subgroup of $\overline{H}$. Let $W$ be the
vector space generated by $D$, then $E\cap W=\{0\}$. Therefore we
define $p_{1}: E\oplus W\longrightarrow E$ the first projection
and $p_{2}: E\oplus W\longrightarrow W$ the second projection.
Now, $H\subset \overline{H}$, then $H=H_{1}+H_{2}$ where
$H_{1}=p_{1}(H)$ and $H_{2}=p_{2}(H)$. Since $E\cap W=\{0\}$, then
$\overline{H}=\overline{H_{1}}+\overline{H_{2}}=E+D$, which yields
that $\overline{H_{1}}=E$ and $H_{2}=D$ because $D$ is closed and
discrete so is $H_{2}$. \end{proof}

\medskip

\begin{proof}[Proof of Corollary~\ref{C:CC01}] The proof follows
directly from Theorem~\ref{t:1}.
\end{proof}
\bigskip

\begin{proof}[Proof of Corollary~\ref{Ct:1}] By Lemma~\ref{L:1}, we can write $\overline{H}=E+D$ with $E$
 is a vector space and $D$ is a discrete additive group  such that $E\oplus W=vect(H)$, where $W:= vect(D)$.
  Then by Theorem~\ref{t:1},  $H=H_{1}+D$ with $H_{1}$ is an additive subgroup dense in $E$. \ \\ (i) If
   $\widetilde{\mathrm{dim}}(H)=p'+ir$ with $p'\leq
dim(E)$ and $r=dim(W)$. As
$\widetilde{\mathrm{dim}}(\overline{H})=dim(E) + i dim(W)$, we
have the results.\ \\ (ii) Write
$\widetilde{\mathrm{dim}}(H)=p+ir$ and
$\widetilde{\mathrm{dim}}(H)=p'+ir'$. Then if $E$ (resp. $E'$) is
the smaller vector space contained in $H$ (resp. $K$) then
$p'=dim(E')\leq p=dim(E)$. Now,  if $W$ (resp. $W'$) is the
smaller vector space containing $H$ (resp. $K$) then
$p'+r'=dim(W')\leq p+r=dim(W)$. It follows that $r'\leq r$ and so
$p'^{2}+r'^{2}\leq p^{2}+r^{2}$.
\end{proof}
\bigskip

\section{{\bf Proof of Theorems ~\ref{t:2},~\ref{t:2} and Corollary ~\ref{c:3}}}

\begin{prop}\label{L:aaaaa1}$($\cite{AM}, Theorem 1.1$)$ Let $u_{1},\dots, u_{p}\in \mathbb{R}^{n}$, ($p\geq n+1$).
Suppose that $(u_{1},\dots, u_{n})$ be a basis of $\mathbb{R}^{n}$
and $u_{k}=\underset{i=1}{\overset{n}{\sum}}\alpha_{k,i}u_{i}$ for
every $n+1\leq k\leq p$. Then the additive group
$H=\underset{k=1}{\overset{p}{\sum}}\mathbb{Z}u_{k}$ is dense in
$\mathbb{R}^{n}$ if and only if the matrix $L(M_{H})=n$.
\end{prop}
\medskip

\medskip

\begin{lem}\label{L:02020} Let $H=\underset{k=1}{\overset{m}{\sum}}\mathbb{Z}u_{k}$, $u_{k}\in
\mathbb{R}^{n}$. Suppose that $vect(H)$ is generated by
$u_{1},\dots, u_{p}$, $p\leq m< n$. Let $v_{p+1},\dots, v_{n}\in
\mathbb{R}^{n}$ such that $(u_{1},\dots,u_{p},v_{p+1},\dots,
v_{n})$ forms a basis of $\mathbb{R}^{n}$. Then for every $1\leq
r\leq n-m$ we have $L(M_{H})=L(M_{H'})$ where $H'=H+
\underset{k=1}{\overset{r}{\sum}}\mathbb{Z}v_{p+k}$. In
particular,
$\Re(\widetilde{\mathrm{dim}}(\overline{H}))=\Re(\widetilde{\mathrm{dim}}(\overline{H'}))$.
\end{lem}
\bigskip

\begin{proof} Write
$u_{k}=\underset{j=1}{\overset{p}{\sum}}\alpha_{k,j}u_{j}$ for
every $p+1\leq k\leq m$. Then $(\alpha_{k,1},\dots, \alpha_{k,p})$
are the coordinate of $u_{k}$ in the basis $(u_{1},\dots,u_{p})$
of $vect(H)$ and $(\alpha_{k,1},\dots, \alpha_{k,p},0,\dots,0)$
are the coordinate of $u_{k}$ in the basis
$(u_{1},\dots,u_{p},v_{p+1},\dots, v_{p+r})$ of $vect(H')$.
Therefore $$M_{H'}=\left[\begin{array}{cc}
  M_{H}, &\  0
\end{array}\right]\in M_{m-p,p+r}(\mathbb{R}).$$ It follows that
$L(M_{H})=L(M_{H'})$. Since
$\Re{(\widetilde{\mathrm{dim}}(\overline{H}))}=dim(\overline{H})$
and by using the definition of $dim(\overline{H})$ as the greatest
dimension of all vector subspaces contained in $\overline{H}$, we
obtain
$\Re(\widetilde{\mathrm{dim}}(\overline{H}))=\Re(\widetilde{\mathrm{dim}}(\overline{H'}))$.
 \end{proof}
 \medskip

\begin{lem}\label{L:02} Let $H=\underset{k=1}{\overset{m}{\sum}}\mathbb{Z}u_{k}$, $u_{k}\in \mathbb{R}^{n}$.
Then for every $P\in GL(n, \mathbb{R})$ we have
$L(M_{H})=L(M_{P(H)})$. In particular
$\widetilde{\mathrm{dim}}(\overline{H})=\widetilde{\mathrm{dim}}(\overline{P(H)})$.
\end{lem}
\bigskip

\begin{proof} Suppose that $(u_{1},\dots,u_{p})$ is a basis of
$vect(H)$ with $1\leq p\leq m$ and write
$u_{k}=\underset{j=1}{\overset{p}{\sum}}\alpha_{k,j}u_{j}$ for
every $p+1\leq k\leq m$. Then
$Pu_{k}=\underset{j=1}{\overset{p}{\sum}}\alpha_{k,j}Pu_{j}$. It
follows that $u_{k}$ and $Pu_{k}$ have the same coordinate
respectively in the basis $(u_{1},\dots,u_{p})$ and
$(Pu_{1},\dots,Pu_{p})$. We conclude that $M_{H}=M_{P(H)}$.
Moreover, $P$ is viewed as an isomorphism, so
$\widetilde{\mathrm{dim}}(\overline{H})=\widetilde{\mathrm{dim}}(\overline{P(H)})$,
because $dim(vect(H))=dim(vect(P(H))$ and
$dim(vect(E))=dim(vect(P(E))$, where $E$ is the greater vector
space contained in $\overline{H}$.
 \end{proof}
 \medskip

\begin{proof}[Proof of Theorem~\ref{t:2}]  By  Lemma ~\ref{L:1}, we can write
$\overline{H}=E+D$ with $E$ is a vector space and $D$ a discrete
additive subgroup of $\overline{H}$. Let $W$ be the vector space
generated by $D$, then $E\cap W=\{0\}$. Denote by
$\widetilde{\mathrm{dim}}(\overline{H})=p+ir$. Since
$H=\underset{k=1}{\overset{m}{\sum}}\mathbb{Z}u_{k}$ then there
exists a basis $\mathcal{B}:=(v_{1},\dots, v_{n})$ of
$\mathbb{R}^{n}$ such that $(v_{1},\dots, v_{p})$ forms a basis of
$E$ and $D=\underset{k=1}{\overset{r}{\sum}}\mathbb{Z}v_{p+k}$.
Denote by $P\in GL(n, \mathbb{R})$ the matrix of basis change from
the canonical basis to $\mathcal{B}$ and $H'=P(H)$. By Lemma
~\ref{L:02}, $H=H_{1}+D$ with $H_{1}$ is an additive group dense
in $E$ (i.e. $\overline{H_{1}}=E$). Write $w_{k}=Pu_{k}$ for every
$k=1,\dots, m$, so
$H'=\underset{k=1}{\overset{m}{\sum}}\mathbb{Z}w_{k}$. Write
$P(H_{1})=\underset{k=1}{\overset{p'}{\sum}}\mathbb{Z}w_{i_{k}}$.
Then $\overline{P(H_{1})}=P(E)$, so by Proposition~\ref{L:aaaaa1},
$L(M_{P(H_{1})})=p$ and by lemma~\ref{L:02}, $L(M_{H_{1}})=p$,
hence by Lemma ~\ref{L:02020}, $L(M_{H})=p$. Since $r=q-p$ with
$q=dim(vect(H))$ then
$\widetilde{\mathrm{dim}}(H)=L(M_{H})+i(q-L(M_{H}))$.
 \end{proof}

\medskip

\begin{prop}\label{pp:2} Let $(u_{1},\dots, u_{n})$ be a basis of $\mathbb{R}^{n}$ then there exists $u\in \mathbb{R}^{n}$ such that
$H:=\underset{k=1}{\overset{n}{\sum}}\mathbb{Z}u_{k}+\mathbb{Z}u$
is dense in $\mathbb{R}^{n}$.
\end{prop}
\bigskip

\begin{proof} let $\alpha_{1},\dots, \alpha_{n}\in \mathbb{R}$
such that $1,\alpha_{1},\dots, \alpha_{n}$ are rationally
independent and
$u=\underset{k=1}{\overset{n}{\sum}}\alpha_{k}u_{k}$. By
Lemma~\ref{L:02}, $L(M_{H})$ is invariant by basis change, then in
the basis $(u_{1},\dots, u_{n})$ we have
$M_{H}=\left[\alpha_{1},\dots, \alpha_{n}\right]$, so
$L(M_{H})=n$. By applying Proposition~\ref{L:aaaaa1},
$\overline{H}=\mathbb{R}^{n}$.
\end{proof}
\medskip

\begin{proof}[Proof of Theorem~\ref{t:3}] Suppose that $(u_{1},\dots, u_{q})$ is a
 basis of $vect(H)$ and let $\alpha_{1},\dots,$ $ \alpha_{q}\in\mathbb{R}$
such that $1,\alpha_{1},\dots, \alpha_{q}$ are rationally
independent and
$u=\underset{k=1}{\overset{q}{\sum}}\alpha_{k}u_{k}$. Denote by
$H'=
\underset{k=1}{\overset{q}{\sum}}\mathbb{Z}u_{k}+\mathbb{Z}u$. By
Proposition~\ref{pp:2}, $\overline{H'}=vect(H)$.
\end{proof}
\medskip

\begin{proof}[Proof of Corollary~\ref{c:3}] Let $u\in \mathbb{R}^{n}$ and $H'=H+\mathbb{Z}u$. Since
$dim(vect(H))=p+r<n$ then $dim(vect(H'))\leq p+r+1\leq n$. Denote
by $dim(H')=p'+ir'$ with $p'\geq p$ and $r'\leq p+r+1-p'$. Then
there are two cases:\
\\ - if $r'\neq 0$ then by Corollary ~\ref{C:CC01}, $H'$
can not be dense in $\mathbb{R}^{n}$.\ \\ - if $r'=0$, so
$p'=p+r+1=n$ then $\mathbb{R}u\oplus vect(H)=\mathbb{R}^{n}$.
Denote by $p_{1}: \mathbb{R}u\oplus vect(H)\longrightarrow
\mathbb{R}u$ the first projection. Then $p_{1}(H')=\mathbb{Z}u$,
so $\mathbb{R}u=p_{1}(\overline{H'})\subset
\overline{p_{1}(H')}=\mathbb{Z}u$, a contradiction.
\end{proof}
\medskip

\section{{\bf Proof of Theorem ~\ref{t:5}}}

\begin{lem}\label{L:0033} Let $H$ and $K$ be two closed additive
subgroup of $\mathbb{R}^{n}$ and $f:H\longrightarrow K$ be a
homomorphism of closed additive group. Then:\ \\ (i) $f(H)$ is a
closed additive group of $\mathbb{R}^{n}$.\ \\ (ii) $f^{-1}(K)$ is
a closed additive group of $\mathbb{R}^{n}$.
\end{lem}

\begin{proof} By Lemma~\ref{L:1} we can write $H=E+D$ and
$K=E'+D'$ with $E, E'$ are two vector spaces with dimension
respectively $p, \ p'$ and $D, D'$ are two discrete additive
groups with dimension respectively $ir, \ ir'$ such that $E\oplus
vect(D)=vect(H)$ and $E'\oplus vect(D')=vect(K)$. Write
$f=f_{1}\oplus f_{2}$ given by $f(x+y)=f_{1}(x)+f_{2}(y)$ for
every $x\in E$ and $y\in D$ with $f_{1}$ is linear and $f_{2}$ is
an homomorphism of additive group.
\
\\
 (i) The proof follows directly from the fact $f_{1}(E)\subset E'$ is  a vector
 space, $f^{-1}_{2}(D')\subset D$ is an additive group and
$f(E+D)=f_{1}(E)+f_{2}(D)$ is a closed additive group.\ \\ (ii)
The proof follows directly from the fact $f^{-1}_{1}(E')\subset E$
is a vector
 space, $f_{2}(D)\subset D'$ is an additive group and
$f^{-1}(E'+D')=f^{-1}_{1}(E)+f^{-1}_{2}(D)$ is a closed additive
group.\
\end{proof}
\bigskip

\begin{cor}\label{L:00443} Let $H$ and $K$ be two closed additive
subgroup of $\mathbb{R}^{n}$ and $f:H\longrightarrow K$ be a
homomorphism of closed additive group. Then:\ \\ (i) $Ker(f)$ is a
closed additive subgroup of $\mathbb{R}^{n}$.\ \\ (ii) $Im(f)$  is
a closed additive subgroup of $\mathbb{R}^{n}$.
\end{cor}
\medskip

\begin{proof} The proof results directly from Lemma~\ref{L:0033},
since $Ker(f)=f^{-1}(0)$ and $Im(f)=f(H)$.
\end{proof}
\bigskip

Let $H$ and $K$  be two additive subgroups of $\mathbb{R}^{n}$. We
say that the algebraic sum of $H$ and $K$ is direct, denoted by
$H\oplus K=\mathbb{R}^{n}$, if $H+K=\mathbb{R}^{n}$ and
$\Re(\widetilde{\mathrm{dim}}(H\cap K))=0$. The group $K$ is said
the algebraic supplement of $H$ in $\mathbb{R}^{n}$  with defect
$s:=\Im(\widetilde{\mathrm{dim}}(H\cap K))$. We say that $K$ is an
algebraic  supplement of $H$ with defect $0$ if
$\Im(\widetilde{\mathrm{dim}}(H\cap K))=0$. in this case, we have
$H\cap K=\{0\}$. Let $H$ be a closed additive subgroup of
$\mathbb{R}^{n}$. Denote by $H=E+D$ with $E$ is a vector space
equipped by the usual topology and $D$ is a discrete group.
\bigskip

\begin{prop}\label{L:06543} Let $H$ and $K$ be two closed additive
subgroup of $\mathbb{R}^{n}$ and $f:H\longrightarrow K$ be a
homomorphism of closed additive group. Then:\ \\ (i) if $f$ is
injective then $f:H\longrightarrow f(H)$ is an isomorphism of
closed additive group of $\mathbb{R}^{n}$.\
\\ (ii) if $f$ is
surjective and $F$ is a closed additive group supplement  of
$Ker(f)$ in $H$ with defect $0$, then the restriction $f_{/F}:
F\longrightarrow K$ of $f$ to $F$ is an isomorphism of closed
additive group of $\mathbb{R}^{n}$. Moreover,
$\Re{(\widetilde{\mathrm{dim}}(K))}=\Re{(\widetilde{\mathrm{dim}}(H))}-
\Re{(\widetilde{\mathrm{dim}}(Ker(f)))}$.
\end{prop}
\medskip

\begin{proof} (i) It is clear that $f:H\longrightarrow f(H)$ is
invertible. Write $H=E+D$ as above and $f=f_{1}\oplus f_{2}$ with
$f_{1}:E\longrightarrow f_{1}(E)$ is linear and $f_{2}:
D\longrightarrow f_{2}(D)$ is an homomorphism of group. Then
$f_{1}$ and $f_{2}$ are also invertible and
$f^{-1}=f^{-1}_{1}\oplus f^{-1}_{2}$. It follows that $f^{-1}$ is
an homomorphism of closed additive group.\ \\ (ii) Let $F$ is a
closed additive group supplement  of $Ker(f)$ in $H$ with defect
$0$. Then $$\Re{(\widetilde{\mathrm{dim}}(Ker(f)\cap
F))}=\Im{(\widetilde{\mathrm{dim}}(Ker(f)\cap F)))}=0,$$ so
$Ker(f)\cap F=\{0\}$. It follows that $f_{/F}$ is injective, so it
is invertible because it is surjective. On the other hand, we have
$\Re{(\widetilde{\mathrm{dim}}(K))}=\Re{(\widetilde{\mathrm{dim}}(F))}
=\Re{(\widetilde{\mathrm{dim}}(H))}-\Re{(\widetilde{\mathrm{dim}}(Ker(f)))}$.
\end{proof}
\medskip

\begin{proof}[Proof of Theorem~\ref{t:5}] The proof of (iv) and (v)
results from Lemma~\ref{L:0033}. \ \\ Proof of (iii) The proof
follows directly from the fact that $f$ is an isomorphism of
closed additive group.

 \ \\ Proof of
(i): Since $f$ is injective then by Proposition~\ref{L:06543},(i)
we have $f:H\longrightarrow f(H)$ is an isomorphism of closed
additive group. Then by (iii),
$\widetilde{\mathrm{dim}}(H)=\widetilde{\mathrm{dim}}(f(H))$. By
Lemma ~\ref{L:0033},(i), $f(H)$ is a closed additive subgroup of
$K$, so by Corollary~\ref{Ct:1},
$|\widetilde{\mathrm{dim}}(f(H))|\leq
|\widetilde{\mathrm{dim}}(K)|$.\ \\   Proof of (i): Since $f$ is
injective then by Proposition~\ref{L:06543},(i) we have
$f:H\longrightarrow f(H)$ is an isomorphism of closed additive
group. Then by (iii),
$\widetilde{\mathrm{dim}}(H)=\widetilde{\mathrm{dim}}(f(H))$. By
Lemma ~\ref{L:0033},(i), $f(H)$ is a closed additive subgroup of
$K$, so by Corollary~\ref{Ct:1},
$|\widetilde{\mathrm{dim}}(f(H))|\leq
|\widetilde{\mathrm{dim}}(K)|$. \ \\   Proof of (ii): Since $f$ is
surjective then by Proposition~\ref{L:06543},(ii), for every
supplement $F$ of $Ker(f)$ in $H$ with 0 defect, we have
$f_{F}:F\longrightarrow K$ is an isomorphism of closed additive
group. Then by (iii),
$\widetilde{\mathrm{dim}}(F)=\widetilde{\mathrm{dim}}(K)$. By
Corollary~\ref{Ct:1}, $|\widetilde{\mathrm{dim}}(F)|\leq
|\widetilde{\mathrm{dim}}(H)|$, So
$|\widetilde{\mathrm{dim}}(K)|\leq |\widetilde{\mathrm{dim}}(H)|$.
\end{proof}
\bigskip

\bibliographystyle{amsplain}
\vskip 0,4 cm

\end{document}